\documentclass[a4paper]{article}

\usepackage{amsmath,amsthm,verbatim,amssymb,amsfonts,hyperref,graphicx,wrapfig,mathtools,nameref,tabstackengine,bbold,color,textgreek}
\usepackage[utf8]{inputenc}

\theoremstyle{plain}
\newtheorem{theorem}{Theorem}
\newtheorem{corollary}{Corollary}
\newtheorem{lemma}{Lemma}

\theoremstyle{definition}
\newtheorem{definition}{Definition}
\newtheorem{example}{Example}
\newtheorem{notation}{Notation}

\newcommand{\br}[1]{\left( #1 \right)}

\newcommand{\set}[1]{\left\lbrace #1 \right\rbrace}

\newcommand{\abs}[1]{\left\lvert #1 \right\rvert}
\newcommand{\norm}[1]{\left\lVert #1 \right\rVert}
\newcommand{\sbr}[1]{\left[ #1 \right]}
\newcommand{\seg}[1]{\sbr{ #1 }}
\newcommand{\segco}[1]{\left[ #1 \right)}
\newcommand{\segoc}[1]{\left( #1 \right]}

\newcommand{\D}{\;\mathrm{d}}

\title{An Improved Convergence Case for Diophantine Approximations on IFS Fractals}
\author{Itamar Cohen-Matalon}
\date{\today}

\begin{document}

\maketitle

\section{Introduction}
{
    The objective of this paper is to (partially) address the issue of finding an analogue to Khintchine’s theorem for IFS Fractals. We study the convergence case for Diophantine approximations, and show an improved result for higher dimensions.
    This matter has been previously studied by Kleinbock, Lindenstrauss and Weiss in \cite{klw} and by Pollington and Velani in \cite{pv}. \cite{pv} show a similar result to the one in this paper (a Khinchine convergence case) and we shall show how our result is an improvement in the higher dimensional cases. A recent paper by Khalil and Luethi \cite{kl} shows a full Khinchine theorem for a certain class of fractals, but here we try to address a more general class of fractals.
    In this paper, we shall use use the translation into the world of random walks, in the spirit of Simmons and Weiss \cite{sw}, and use results by Benoist and Quint from \cite{bq2} and \cite{bq3}, and Khalil's \cite{khalil} that shall be central to our proof. To introduce our result, we first need to define a few important parameters:
    \begin{definition}[\cite{khalil} (1.3) \& (6.1)]
        For $1 \le k \le d$ we denote by $\mathcal{A}\br{d, k}$ the set of all of the affine subspaces of dimension $k$ in $\mathbb{R}^d$. For $\mathcal{L} \in \mathcal{A}\br{d, k}$ and $\varepsilon > 0$ we denote by $\mathcal{L}^{\br{\varepsilon}}$ the open $\varepsilon$-neighborhood of $\mathcal{L}$ with respect to the Euclidean metric. Given a compactly supported Borel measure $\mu$ on $\mathbb{R}^d$ and $1\le l \le d$ we define:
        \begin{equation}\label{def:alpha}
            \alpha_l\br{\mu} := \liminf\limits_{\varepsilon \rightarrow 0} \frac{\log \sup_{\mathcal{L} \in \mathcal{A}\br{d, d-l}} \mu \br{\mathcal{L}^{\br{\varepsilon}}}}{\log\br{\varepsilon}}.
        \end{equation}
        We shall then define:
        \begin{equation}
            \varpi := \min_{1 \le l \le d} \alpha \br{l}\br{d - l + 1}.
        \end{equation}
    \end{definition}
    Our main result is as follows (see Section \ref{section:def} where we present the general setting and definitions):
    \begin{theorem}\label{theorem:main}
    	Let $\mathcal{K} \subseteq \mathbb{R}^d$ be the limit set of an irreducible, finite system of contracting similarity maps, with a common contraction ratio $0 < \kappa < 1$, satisfying the open set condition and let $\mu_{\mathcal{K}}$ be a Bernoulli measure on $\mathcal{K}$.
    	Let $0 < \alpha < \varpi$, $\psi: \mathbb{N} \rightarrow \mathbb{R}$ monotone decreasing such that:
        \begin{equation}
            \sum_{x=1}^{\infty} x^{\frac{\alpha}{d}-1}\psi^{\alpha}\br{x} < \infty,
        \end{equation}
        then $\mu_{\mathcal{K}}$-a.e. $\mathbf{x} \in \mathcal{K}$ is not $\psi$-approximable.
    \end{theorem}
    An interesting example we shall observe, is the product of Cantor sets, for which we shall calculate the exact constants:
    \begin{lemma}\label{lemma:consts}
        Let $d \ge 1$. Let $\mathcal{K} = \mathcal{C}^d$ where $\mathcal{C}$ is the standard middle thirds Cantor set. Let $s = \operatorname{dim}_H \br{\mathcal{K}}$, and let $\mu_{\mathcal{K}}$ be the restriction to $\mathcal{K}$ of the $s$-dimensional Hausdorff measure. Then $\varpi = s = d \frac{\log 2}{\log 3}$.
    \end{lemma}
    and this shall give us the the specific result:
    \begin{corollary}
        For $\mathcal{K} = \mathcal{C}^d$ and $\mu_{\mathcal{K}}$ as in Lemma \ref{lemma:consts}, the result of the theorem holds for $0 < \alpha < \varpi = d \frac{\log 2}{\log 3}$.
    \end{corollary}
    This result improves upon a previous result by Pollington \& Velani in \cite{pv}. The Pollington-Velani result states, in our notation, that if the following series converges:
    \begin{equation}
        \sum_{x=1}^{\infty} x^{\frac{\alpha_1}{d}-1} {\psi\br{x}}^{\alpha_1},
    \end{equation}
    then $\mu_{\mathcal{K}}\br{\psi-approximable} = 0$, where $\alpha_1 = \alpha_1\br{\mu_{\mathcal{K}}}$ as defined in $\br{\ref{def:alpha}}$.
    \par Our result from Theorem \ref{theorem:main} states that for $0< \alpha < \varpi$, if the following series converges:
    \begin{equation}
        \sum_{x=1}^{\infty} x^{\frac{\alpha}{d}-1} {\psi\br{x}}^{\alpha},
    \end{equation}
    then $\mu_{\mathcal{K}}\br{\psi-approximable} = 0$. Taking as an example the Cantor set product $\mathcal{C}^d$ demonstrates the strength of our result. When $d > 1$ we have $\varpi = d\frac{\log 2}{\log 3} > \frac{\log{2}}{\log{3}} = \alpha_1$, making our result a tighter one.
    When $d=1$, $\varpi = \alpha_1$, our result is almost as good, as we have $\alpha < \varpi = \alpha_1$, and the result from \cite{pv} can handle more general functions.
    \par We believe that with some additional work it should be possible to adjust the proof and remove the ``common contraction ratio" limitation from the iterated function system in Theorem \ref{theorem:main}.
    \par Acknowledgements: This work is part of the authors M.Sc. thesis, conducted in Tel Aviv University under the supervision of Barak Weiss. The author acknowledges funding through grants ISF 2919/19 and BSF 2016256.
    
    \subsection{Proof Outline}
    {
        In order to prove Theorem \ref{theorem:main}, we will start by proving bounds on the behaviour of the random walk that corresponds to the IFS that generates $\mathcal{K}$, and later, using a Lemma from \cite{sw} and the Dani Correspondence from Kleinbock \& Margulis's \cite{km99}, translate this result into the required bound on approximations.
        \par For giving a bound on the escape rate of the IFS random walk, we shall use the machinery of Benoist-Quint (\cite{bq2} \& \cite{bq3}) and observe an exponentially recurrent subset $Y \subset X$, to which we shall (with probability 1) return. Using the upgraded contraction hypothesis given by \cite{khalil}, we can give an interesting bound on the length of each excursion from $Y$ and back of the random walk. The bounds on the excursions will translate into escape distance bounds, which we will then be able to translate into the required approximation bounds.
    }
}
\section{Setting and Definitions}\label{section:def}
{
    \begin{definition}
        Let $\psi: \mathbb{N} \rightarrow \segco{0,\infty}$ and $\mathbf{x} \in \mathbb{R}^d$. We shall say $\mathbf{x}$ is \textbf{$\psi$-approximable} if for infinitely many $q \in \mathbb{N}$ there exists some $\mathbf{p} \in \mathbb{Z}^d$ such that $\norm{\mathbf{x} - \frac{\mathbf{p}}{q}}_\infty < \frac{\psi\br{q}}{q}$.
    \end{definition}
    Let $G = \textrm{SL}_{d+1}\br{\mathbb{R}}$, $\Gamma = \textrm{SL}_{d+1}\br{\mathbb{Z}}$, $X = G/\Gamma$. Let $E$ be a finite set, $s \mapsto h_s$ is a map from $E$ to $G$ and $\mu \in \operatorname{Prob}\br{E}$ a measure such that $\operatorname{supp}\br{\mu} = E$. Sometimes we will choose to consider $E \subseteq G$, and then $s = h_s$. We denote $B = E^{\mathbb{N}}$. For $b = \br{s_1, s_2, ...} \in B$ and $n \in \mathbb{N}$, we write $b^n_1 = \br{s_1, ..., s_n}$ and $h_{b^n_1}$ denotes the product $h_{s_n} \cdots h_{s_1}$. $B$ is equipped with the measure $\beta = \mu ^ {\otimes \mathbb{N}}$, and $T: B \rightarrow B$ is the left shift. We shall define a random walk on $X$ using the transition probabilities that are determined by $\mu$ in the following manner: $P_x$ is the pushforward of $\mu$ under the map $s \mapsto h_s x$. I.e.:
    \begin{equation}
        P_x\br{A} = \int_G \mathbb{1}_A \br{h_s x} \D \mu\br{s}.
    \end{equation}
    \par Specifically, we will be interested in finite, contracting Iterated Function Systems (IFS). This is a set of similarity maps which are maps $\Phi = \br{\phi_s}_{s\in E}$, where a map $\phi_s: \mathbb{R}^d\rightarrow\mathbb{R}^d$ is called a similarity map if:
    \begin{equation}
        \phi\br{x} = \kappa_s x O_s + y_s,
    \end{equation}
    where $O_s$ is a $d \times d$ matrix orthogonal with respect to the inner product, $\kappa_s \in \mathbb{R}$, $\kappa_s > 0$ and $y_s \in \mathbb{R}^d$.
    We will require they be \textit{contracting}, i.e. with a \textit{contraction ratio} - $\kappa_s$ that satisfies $\kappa_s < 1$. We will also be limiting ourselves to IFSes with a common contraction ratio for all maps. Notice that we chose $x$ to be a row vector, with matrices acting from the right, as this will be more consistent with the calculations in the paper.
    \par An IFS $\Phi$ is said to satisfy the \textit{open set condition} if $\exists U \subseteq \mathbb{R}^d$ open and non-empty such that $\br{\phi_s\br{U}}_{s\in E}$ is a disjoint collection of subsets of $U$ ($\phi_s\br{U} \subseteq U$ for each $s \in E$). It is also called \textit{irreducible} if there is no proper affine subspace $\mathcal{L} \subsetneq \mathbb{R}^d$ such that:
    \begin{equation}
        \phi_s\br{\mathcal{L}} = \mathcal{L} \quad\quad \forall s \in E.
    \end{equation}
    \begin{definition}[Coding Map]\label{def:coding_map}
        The coding map of $\Phi$ shall be defined as $\pi: B \rightarrow \mathbb{R}^d$ by:
        \begin{equation}
            \pi\br{b} = \lim\limits_{n \rightarrow \infty} \phi_{b^1_n} \br{\alpha_0},
        \end{equation}
        where $\alpha_0 \in \mathbb{R}^d$ is an arbitrary fixed point and:
        \begin{equation}
            \phi_{b^1_n} = \phi_{b_1} \circ \cdots \circ \phi_{b_n}.
        \end{equation}
        The image of $B$ under the coding map $\pi$ is called the \textit{limit set} of $\Phi$. We denote it by $\mathcal{K} = \mathcal{K}\br{\Phi}$. In our setting, of a finite strictly contracting IFS, $\pi$ converges everywhere and is continuous \cite{sw} (chapter 8.1), and therefore the limit set $\mathcal{K}$ is compact.
    \end{definition}
    \begin{definition}[Bernoulli Measure on $\mathcal{K}$]\label{def:bernoulli}
        A measure $\mu_\mathcal{K}$ is called a Bernoulli measure on $\mathcal{K}$ if it is the pushforward of such a measure $\beta = \mu ^ {\otimes \mathbb{N}}$ onto the fractal $\mathcal{K}$ by the coding map $\pi$.
    \end{definition}
    \begin{example}\label{example:hausdorff}
        An interesting example of such a Bernoulli measure is the Hausdorff measure on the fractal. Let $s = \operatorname{dim}_H \br{\mathcal{K}}$, and let $\mu_\mathcal{K}$ be the restriction to $\mathcal{K}$ of the $s$-dimensional Hausdorff measure. A result by Hutchinson shows that $\mu_\mathcal{K}$ is in-fact a Bernoulli measure, resulting when $\mu$, the measure on $E$ is chosen to be uniform \cite{sw} (Chapter 8.1).
    \end{example}
    We shall fix $\Phi = \set{\phi_s}_{s \in E}$ to be an irreducible, finite system of contracting similarity maps on $\mathbb{R}^d$, chosen according to a law $\mu$, with a common contraction ratio $0 < \kappa < 1$, satisfying the open set condition. We construct the random walk defined by $H = \set{h_s}_{s \in E} \subseteq G$ acting on the space $X$ in Chapter \ref{ifs_translation}, where each step is chosen by the same law $\mu$.
    \par We shall denote the limit set by $\mathcal{K} = \mathcal{K}\br{\Phi}$, and the Bernoulli measure which is the projection of $\mu ^ {\otimes \mathbb{N}}$ onto $\mathcal{K}$ by $\mu_{\mathcal{K}}$. Note that $\mu_{\mathcal{K}}$ depends on $\mu$, despite the fact that our notation does not show this.
    \subsection{Translation Between Iterated Function Systems and Random Walks}\label{ifs_translation}
    {
        To those unfamiliar with the notations of \cite{sw}, this chapter will introduce the translation between IFS language and Random Walk language (as done in Chapter 10.1 of \cite{sw}).
        
        \par Given an IFS $\br{\phi_s}_{s \in E}$ our objective is to translate the maps $\phi_s$ into elements $h_s$ that act on $X$. Notice that in the IFS setup, we were interested in $\phi_{b_n^1}$ (see Definition \ref{def:coding_map}), and in the homogeneous space we are interested in the objects $h_{b_1^n} := h_{b_n} \cdots h_{b_1}$. Therefore we would like the relation to be of the form $h_s = \phi_s^{-1}$, so that the coding map agrees with the random walk and we have $h_{b_1^n} = \phi_{b_n^1}^{-1}$. So we need to understand how to view the similarity map as an element of $G$.
        \begin{notation}\label{subgroups}
        Let $\alpha \in \mathbb{R}^d$. We define:
            \begin{equation}
                a_t = 
                \begin{bmatrix}
                    e^{t} & 0     \\
                    0   & e^{-t/d}I_d
                \end{bmatrix}, \quad
                u_\alpha = 
                \begin{bmatrix}
                    1 & -\alpha     \\
                    0 & I_d
                \end{bmatrix},
            \end{equation}
            and $A = \set{a_t: t\in\mathbb{R}}, U = \set{u_\alpha: \alpha \in \mathbb{R}^d}$. Fixing an inner product on $R^d$, we shall define by $O_d$ the group of matrices preserving this inner product, and define:
            \begin{equation}
                K = \set{
                    \begin{bmatrix}
                        1 & 0     \\
                        0 & O
                    \end{bmatrix}
                    \mid O \in O_d
                }.
            \end{equation}
            Finally let $P=AKU < G$. Notice that $A$ and $K$ commute and normalize $U$.
        \end{notation}
        As $A$ and $K$ normalize $U$, we have the natural projections $\pi_A: P \rightarrow A$ and $\pi_K: P \rightarrow K$ which are homeomorphisms. Observe $\iota: \mathbb{R}^d \rightarrow P/AK$ defined by $\iota\br{\alpha} = u_{-\alpha} AK$. $\iota$ is a homeomorphism, and $\iota\br{0}$ is the identity coset $AK \in P/AK$.
        \par Now consider the action $\rho$ of $P$ on $\mathbb{R}^d$ that results from conjugating the action of $P$ on $P/AK$ by left multiplication, by $\iota$.
        \par We'll have that:
        \begin{equation}
            \begin{split}
                \rho\br{u_\alpha} \br{\beta} &= \iota^{-1}\br{u_\alpha\iota\br{\beta}} = \iota^{-1}\br{u_\alpha u_{-\beta} AK} = \iota^{-1}\br{u_{\alpha -\beta}} = \beta - \alpha \\
                \rho\br{a_t} \br{\beta} &= \iota^{-1}\br{a_t \iota\br{\beta}} = \iota^{-1}\br{a_t u_{\beta} AK} = \iota^{-1}\br{a_t u_{\beta} a_{-t} AK} = \iota^{-1}\br{u_{e^{t+t/d}\beta} AK} = e^{t+t/d} \beta \\
                \rho\br{O'} \br{\beta} &= \iota^{-1}\br{O' \iota\br{\beta}} = \iota^{-1}\br{O' u_{\beta} AK} = \iota^{-1}\br{O' u_{\beta} O'^{-1} AK} = \iota^{-1}\br{u_{\beta O^{-1}} AK} = \beta O^{-1},
            \end{split}
        \end{equation}
        where $O' = \begin{bmatrix} 1 & 0 \\ 0 & O \end{bmatrix}$. Notice that $\rho$ is faithful, and $\rho{P}$ is the group of similarity maps on $\mathbb{R}^d$. We can therefore identify $P$ with the group of similarity maps of $\mathbb{R}^d$, allowing the translation into the matrices $h_s$ acting on $X$ in the random walk process by: $h_s = \phi_s^{-1}$.
    }
}
\section{Exponentially Recurrent Subset}
{
    The core of the proof relies on bounding excursions from a specific compact set within $X$, to which the random walk returns with probability 1. The machinery for finding such sets was introduced by Benoist-Quint in \cite{bq2} and \cite{bq3}. In the key result of this chapter - Corollary \ref{cor:from_lemma_63}, we prove the existence of such a family of sets, with specific parameters that shall allow us to prove the bounds on excursions in Chapter \ref{sec:excursion_bounds}.
    \begin{definition}
        Let $Y \subset X$ be a Borel set. We define the \textit{return time} $\tau_{Y,x}:B \rightarrow \mathbb{N}_{\ge 1}\cup\set{\infty}$ for $b \in B$ as:
        \begin{equation}
            \tau_{Y,x}\br{b} = \inf\set{n\ge 1 \mid h_{b^n_1} x \in Y}.
        \end{equation}
    \end{definition}
    \begin{definition}
        We define the $n$-th consecutive return time, for $n \ge 2$ as:
        \begin{equation}
            \tau_{Y, x}^n\br{b} = \inf\set{n > \tau_{Y, x}^{n-1} \mid h_{b^n_1} x \in Y},
        \end{equation}
        and set $\tau_{Y, x}^1 = \tau_{Y, x}$.
    \end{definition}
    We may sometimes drop part of the notation parameters when they are clear from the context, as it will make equations easier to read.
    \begin{notation}
        We denote the sub-level sets of a function $f:X \rightarrow \seg{0, \infty}$ as: $X_M = \set{x: f\br{x} \le M}$ and denote $\tau_{M,x} := \tau_{X_M,x}$.
    \end{notation}
    \begin{definition}[\cite{bq2} - Definition 6.1]
        We say that $Y$ is exponentially $\mu$ recurrent if for some $a_0 < 1$ one has:
        \begin{equation}
            \sup\limits_{x \in Y} \int_B a_0^{-\tau_{Y,x}\br{b}} \D \beta\br{b} < \infty.
        \end{equation}
    \end{definition}
    \begin{definition}
        We denote by $P_{\mu}$ the one step averaging operator of the random walk on $X$ defined by $P\br{f}\br{x} = \int_E f\br{h_s x} \D \mu \br{s}$.
    \end{definition}
    \begin{lemma}[\cite{bq2} - Proposition 6.3]\label{lemma:bq6.3}
        Using the existing notations, let $\mu$ a Borel probability on $G$ and denote by $P_{\mu}$ the one step averaging operator for $\mu$. Let $f: X \rightarrow \seg{0, \infty}$ be a non-negative Borel function satisfying $P_{\mu} f \le a f + c$ for some $0 < a < 1$ and $c > 0$. Let $a < a_0 < 1$ and let $M \ge M_0 = \frac{b+1}{a_0-a}$. Then for all $x \in X$:
        \begin{equation}
            \int_B a_0^{-\tau_{M,x}\br{b}}\D \beta\br{b} \le \max\br{M,f\br{x}}.
        \end{equation}
    \end{lemma}
    Functions satisfying $P_{\mu} f \le a f + c$ for some $0 < a < 1$ and $c > 0$ are called Margulis-Lyapunov functions. Lemma \ref{lemma:bq6.3} allows us to find exponentially recurrent subsets using such functions. The ``quality'' of the exponentially recurrent subset is effectively determined by the parameter $a$. In his paper \cite{khalil}, Khalil constructs a family of such functions, we will use here. We will use a simpler form of the Contraction Hypothesis (\cite{khalil} Definition 3.1) using the simplification in \cite{khalil} Remark 3.2 (1). Specifically, Khalil constructs the Margulis function, and the contraction hypothesis for our setup in section 6 (page 28) of \cite{khalil}.
    \begin{notation}
        Recall the notation of $a_t$ in \ref{subgroups}. We shall denote: 
        \begin{equation}\label{eq:def_gt}
            g_t = a_{-\frac{d \log t}{d+1}}.
        \end{equation}
    \end{notation}
    Notice that unlike $a_t$ which satisfies an additive identity: $a_t a_s = a_{t+s}$, $g_t$ satisfies a multiplicative one: $g_t g_s = g_{ts}$.
    \begin{definition}[Contraction Hypothesis]\label{def:ch}
        Let $X$ and $\mu_{\mathcal{K}}$ be as previously defined. Given a collection of functions $\mathcal{F} = \set{f_m: X \rightarrow \br{0, \infty} : m \in \mathbb{N}}$ and a real number $\alpha > 0$, we say that $\mu_{\mathcal{K}}$ satisfies the $\br{\mathcal{F}, \alpha}$\textbf{-contraction hypothesis} on $X$ if the following properties hold:
        \begin{enumerate}
            \item The functions $f_m$ are proper for all $m \in \mathbb{N}$.
            \item For every $m \in \mathbb{N}$, $f_m$ is $\operatorname{SO}\br{d+1, \mathbb{R}}$-invariant and uniformly log Lipschitz with respect to the $G$ action. That is, for every bounded neighborhood $\mathcal{O} \subseteq G$ of the identity, there exists a constant $C_{\mathcal{O}} \ge 1$ such that for every $g \in \mathcal{O}$, $x \in X$ and $m \in \mathbb{N}$:
            \begin{equation}
                C_{\mathcal{O}}^{-1} f_m\br{x} \le f_m\br{gx} \le C_{\mathcal{O}} f_m\br{x}.
            \end{equation}
            \item There exists a constant $c \ge 1$ such that the following holds: for all $m \in \mathbb{N}$, there exists $T > 0$ such that for all $y \in X$:
            \begin{equation}\label{def:ch:3}
                \int_\mathcal{K} f_m \br{g_{t_m} u_{\mathbf{x}} y} \D \mu_{\mathcal{K}}\br{\mathbf{x}} \le A_m f_m \br{y} + B_m,
            \end{equation}
            where $t_m = \kappa^m$, $A_m = c\kappa^{m\alpha}$ and $B_m = T\kappa^{-m\alpha}$.
        \end{enumerate}
    \end{definition}
    \begin{theorem}[Contraction Hypothesis for $\mu_{\mathcal{K}}$]\label{thm:ch}
        Let $0 < \varrho < 1$. There exists a family of functions $\mathcal{F}$ such that $\mu_{\mathcal{K}}$ satisfies the $\br{\mathcal{F}, \alpha}$-contraction hypothesis with:
        \begin{equation}\label{eq:def_alpha}
            \alpha = \frac{\varrho \varpi}{d+1}.
        \end{equation}
    \end{theorem}
    We would like to use this result, together with the Benoist-Quint machinery, in order to prove the existence of an exponentially recurrent set, with favorable parameters. This is achieved in Corollary \ref{cor:from_lemma_63}. 
    To do this, we first need a translation between the trajectory of the diagonal flow and the random walk, coming from an important lemma from \cite{sw}, which is Lemma \ref{lemma:bdd_dist} in this paper.
    \begin{definition}\label{def:bdd_dist}
        We say a trajectory $\br{x_n}_{n \in \mathbb{N}} \subseteq X$ \textbf{remains within bounded distance} of a trajectory $\br{y_n}_{n \in \mathbb{N}} \subseteq X$ if there exists some bounded neighborhood of the identity, $\mathcal{O} \subseteq G$, such that for each $n \in \mathbb{N}$ we can find some $g \in \mathcal{O}$ such that:
        \begin{equation}
            x_n = g y_n.
        \end{equation}
        We shall say this bound is \textbf{independent} of some parameters if we can find a common $\mathcal{O} \subseteq G$ that satisfies the requirements for each choice of parameters.
    \end{definition}
    
    \begin{lemma}\label{lemma:bdd_dist}
        For all $b \in B$, the trajectory of the diagonal flow $\br{a_{t} u_{\pi\br{b}} x_0}_{t > 0}$ remains bounded distance from the random walk trajectory $\br{h_{b_1^n} x_0}_{n \in \mathbb{N}}$. This bound is independent of $n$, $x_0$ and $b$. Specifically, $a_{t_n} u_{\pi\br{b}} x_0$ remains bounded distance from $h_{b_1^n} x_0$
        where $a_{t_n}$ is the projection onto the diagonal group element of $h_{b_1^n}$.
    \end{lemma}
    This Lemma can be deduced from the proof of Lemma 11.3 in \cite{sw}. The proof can be found in Chapter \ref{appendix}, together with the notations that translate IFSs into Random Walks. We would like to continue by calculating $t_n$ for this ``translation''. The IFS is composed of elements of the form:
    \begin{equation}
        \phi_{O, b}\br{v} = \kappa\br{v O + b},
    \end{equation}
    where $b \in \mathbb{R}^d$, $O \in \operatorname{SO}\br{d, \mathbb{R}}$ and $\kappa \in \mathbb{R}$. Taking such $\phi = \phi_{O, b}$ we would like calculate the projection onto the diagonal group of the corresponding element $h \in H$, which we'll denote by $a_t$ (as in \cite{sw} chapter 10.1, where it is shown that the projection is also a group homomorphism). This is determined by $\kappa$, so it is identical for all the elements of $H$. As $h$ corresponds to the inverse of $\phi$ (Chapter \ref{ifs_translation}), the diagonal element of $\phi$ is $a_{-t}$. We therefore have that: $\kappa = e^{-t/M - t/N}$, where in our case $M=1, N=d$. So:
    \begin{equation}
        \kappa = e^{-t/M-t/N} = e^{-\frac{d+1}{d}t} \quad \Rightarrow \quad \log \kappa = -\frac{d+1}{d}t \quad \Rightarrow \quad t = -\frac{d\log \kappa}{d+1}.
    \end{equation}
    As the projection to the diagonal group is the product of the projections of the $n$ acting elements, each contributing $a_t$, the total projection is $a_{nt}$, and so:
    \begin{equation}\label{eq:tn_formula}
        t_n = nt = -n\frac{d\log \kappa}{d+1}.
    \end{equation}
    Notice that this projection can be written as $g_{\kappa^n}$ by (\ref{eq:def_gt}).
    
    \begin{corollary}\label{cor:C}
        In the same notation above, and $\alpha$ as in (\ref{eq:def_alpha}) we have that there exists some $C \ge 1$ such that for all $m \in \mathbb{N}$ and for all $y \in X$:
        \begin{equation}
            P^m f_m \br{y} \le C A_m f_m\br{y} + C B_m.
        \end{equation}
    \end{corollary}
    \begin{proof}
        Clearly:
        \begin{equation}
            P^m f_m \br{y} = \int_{B} f_m\br{h_{b^n_1} y} \D \beta\br{b}.
        \end{equation}
        Denote by $b^{\mathbf{x}}_n$ the prefix of length $n$ of the path corresponding to $\mathbf{x}$, and notice that for $\mathbf{x} = \pi\br{b}$ we have the identity: $b^n_1 = b^{\mathbf{\pi\br{b}}}_n$. As $\mathcal{K}$ is the image of $B$ under the coding map, and $\mu_{\mathcal{K}}$ is the pushforward of $\beta$ (Definitions \ref{def:coding_map} and \ref{def:bernoulli}) we have that:
        \begin{equation}
            P^m f_m \br{y} = \int_{B} f_m\br{h_{b^n_1} y} \D \beta\br{b} = \int_{\mathcal{K}} f_m\br{h_{b^{\mathbf{x}}_m} y} \D \mu_{\mathcal{K}}\br{\mathbf{x}}.
        \end{equation}
        By Lemma \ref{lemma:bdd_dist}, there is some compact neighborhood of the identity, $\mathcal{O} \subseteq G$ (independent of $b$) such for any $n$ and $y$ there exists some $g \in \mathcal{O}$ such that $h_{b^{\mathbf{x}}_n} y = g a_{t_n} u_{\mathbf{x}} y$, and as $a_{t_n} = g_{\kappa^n}$, we have the integrand in Formula (\ref{def:ch:3}). By (2) of the Definition \ref{def:ch}, $f_m\br{h_{b^{\mathbf{x}}_n} y} \le C_{\mathcal{O}} f_m\br{g_{\kappa^n} u_{\mathbf{x}} y}$, and so:
        \begin{multline}
            P^m f_m \br{y} = \int_{\mathcal{K}} f_m\br{h_{b^{\mathbf{x}}_m} y} \D \mu_{\mathcal{K}}\br{\mathbf{x}} \le \int_{\mathcal{K}} C_{\mathcal{O}} f_m\br{g_{\kappa^m} u_{\mathbf{x}} y} \D \mu_{\mathcal{K}}\br{\mathbf{x}} = \\ = C_{\mathcal{O}} \int_\mathcal{K} f_m \br{g_{\kappa^m} u_{\mathbf{x}} y} \D \mu_{\mathcal{K}}\br{\mathbf{x}} \le C_{\mathcal{O}} \br{A_m f_m \br{y} + B_m} = C_{\mathcal{O}}A_m f_m \br{y} + C_{\mathcal{O}}B_m,
        \end{multline}
        as required.
    \end{proof}
    Now that we see $f_m$ are in the form required by Lemma \ref{lemma:bq6.3}, but we still have one more requirement before we can use the lemma:
    \begin{corollary}\label{cor:m_0}
        Let $C$ be as in Corollary \ref{cor:C}. There exists a large enough $m_0$, such that for any $m \ge m_0$ $CA_m < 1$.
    \end{corollary}
    \begin{proof}
        As by Definition \ref{def:ch} part (3), $A_m = c\kappa^{m\alpha}$, and as $\kappa < 1$ and $\alpha > 0$, $A_m$ is monotone decreasing in $m$ (with limit 0 in infinity), so for large enough values, $CA_m < 1$.
    \end{proof}
    We can now apply Lemma \ref{lemma:bq6.3}:
    \begin{corollary}\label{cor:from_lemma_63}
        Let $C$ be as in Corollary \ref{cor:C} and let $m_0$ be as in Corollary \ref{cor:m_0}. For every $m \ge m_0$, $\delta \in \br{0, -\log \br{CA_m}}$ there exists a compact $Y \subseteq X$ such that:
        \begin{equation}
            \sup\limits_{x \in Y} \int_B e^{\frac{\delta}{m} \tau\br{b}} \D \beta \br{b} < \infty,
        \end{equation}
        where $\tau = \tau_{Y, x}$.
    \end{corollary}
    \begin{proof}
        \par Taking $m \ge m_0$ we denote by $\tau_m = \tau_{Y, x, m}$ the first return time to $Y \subseteq X$, along the sampled times (every m steps). Using Lemma \ref{lemma:bq6.3} we have that for any $CA_m = a < a_0 < 1$ there exists some $M_m = M_m\br{a_0} > 0$ for $f_m$ such that $X_{M_m}$ is exponentially recurrent, as:
        \begin{equation}
            \sup\limits_{x \in X_{M_m}} \int_B a_0^{-\tau_{x, m}\br{b}}\D \beta\br{b} \le \sup\limits_{x \in X_{M_m}} \max\br{M_m,f_m\br{x}} \le M_m,
        \end{equation}
        denoting: $\tau_{x, m} := \tau_{X_{M_m},x,m}$.
        \par Notice, that as $f_m$ is proper, $X_{M_m}$ is compact. Changing variables to $\delta = \log{a_0}$ we have that for any $\delta \in \br{0, -\log \br{CA_m}}$ exists a large enough $M = M\br{\delta} = M_m\br{a_0}$, so that taking $Y = X_M$ gives us:
        \begin{equation}
            \operatorname{sup}\limits_{x \in Y} \int_B e^{\delta \tau_{Y, x, m}\br{b}} \D \beta \br{b} \le M < \infty.
        \end{equation}
        As $\tau_{Y, x, m}$ is the first return time to $Y$ sampled every m steps it corresponds to time $m\tau_{Y, x, m}$ along the full random walk. We therefore have $\tau_{Y,x} \le m\tau_{Y, x, m}$. Substituting gives the required result.
    \end{proof}
}
\section{Excursion Bounds}\label{sec:excursion_bounds}
{
    Using the family of exponentially recurrent subsets from Corollary \ref{cor:from_lemma_63}, we shall now give bounds on the length of the excursions from these sets. These bounds are achieved in Lemma \ref{lemma:excursion_bound} and Corollary \ref{cor:excursion_bound}.
    \begin{definition}
        For $n \ge 1$ define the $n$-th excursion length by: $\sigma_{Y, x}^n = \tau_{Y, x}^{n+1} - \tau_{Y, x}^{n}$, and: $\sigma_{Y, x}^0 = \tau_{Y, x}^{1}$. We can view $\sigma_{Y, x}^n$ as a random variable that depends on the trajectory.
    \end{definition}
    \begin{definition}
        We shorthand $\tau\br{x,n} := \tau_{Y, x}^n$ and $h_{b,k} := h_{b_1^k}$ denote:
        \begin{equation}
            y_n\br{x, b} := h_{b, \tau\br{x,n}} x \quad\quad \br{= h_{b_1^{\tau_{Y, x}^n}} x},
        \end{equation}
        i.e. the $n$-th return point in $Y$.
    \end{definition}
    We continue by proving a Lemma \& Corollary inspired by \cite{bq3} - Lemma 3.5.
    \begin{lemma}\label{lemma:excursion_bound}
        Fix $x \in X$ and $r: \mathbb{N} \rightarrow \segco{0, \infty}$. Taking $Y$ and $\delta$ as in  Corollary \ref{cor:from_lemma_63} and denoting $\theta = \operatorname{sup}\limits_{x \in Y} \int_B e^{\frac{\delta}{m} \tau_{Y, x}\br{b}}$, we have that for all $n \in \mathbb{N}$ :
        \begin{equation}
            \beta\br{\sigma_{Y, x}^n \ge r\br{n}} \le e^{-\frac{\delta}{m} r\br{n}} \theta.
        \end{equation}
    \end{lemma}
    \begin{proof}
        First we use Chebyshev's inequality:
        \begin{equation}
            \beta\br{\sigma_{Y, x}^n \ge r\br{n}} \le 
            e^{-\frac{\delta}{m} r\br{n}} \int_B e^{\frac{\delta}{m} \sigma_{Y, x}^n\br{b}} \D \beta\br{b}.
        \end{equation}    
        Notice that $\sigma_{Y, x}^n = \tau_{Y, y_{n}\br{x, b}}\br{T^{\tau_{Y, x}^{n}} b}$, i.e. the first return time from the previous point in $Y$, along the shifted sequence $b$. As $y_{n}\br{x, b}$ depends only on $\br{b_1, ..., b_{\tau_{Y, x}^{n}}}$ and $T^{\tau_{Y, x}^{n}} b = \br{b_{\tau_{Y, x}^{n}+1}, ...}$, these are independent, and so we can write:
        \begin{equation}
            \int_B e^{\frac{\delta}{m} \sigma_{Y, x}^n \br{b}} \D \beta\br{b} = 
            \int_B e^{\frac{\delta}{m} \tau_{Y, y_{n}\br{x, b}}\br{T^{\tau_{Y, x}^{n}} b}} \D \beta\br{b} = 
            \int_B \int_B e^{\frac{\delta}{m} \tau_{Y, y_{n}\br{x, b}}\br{b'}} \D \beta\br{b'} \D \beta\br{b}.
        \end{equation}
        As $y_{n}\br{x, b} \in Y$, we have that $\int_B e^{\frac{\delta}{m} \tau_{Y, y_{n}\br{x, b}}\br{b'}} \D \beta\br{b'} \le \theta$, and so:
        \begin{equation}
            \int_B e^{\frac{\delta}{m} \sigma_{Y, x}^n \br{b}} \D \beta\br{b} = 
            \int_B \int_B e^{\frac{\delta}{m} \tau_{Y, y_{n}\br{x, b}}\br{b'}} \D \beta\br{b'} \D \beta\br{b} \le 
            \int_B \theta \D \beta\br{b} = \theta.
        \end{equation}
        Finally we get:
        \begin{equation}
            \beta\br{\sigma_{Y, x}^n \ge r\br{n}} \le 
            e^{-\frac{\delta}{m} r\br{n}} \int_B e^{\frac{\delta}{m} \sigma_{Y, x}^n\br{b}} \D \beta\br{b} \le 
            e^{-\frac{\delta}{m} r\br{n}} \theta,
        \end{equation}
        as required.
    \end{proof}
    \begin{corollary}\label{cor:excursion_bound}
        In the above notation, if $\sum_{n=1}^{\infty} e^{-\frac{\delta}{m} r\br{n}} < \infty$, then for $\beta$-a.e. $b \in B$ there exists a large enough $n_0 = n_0\br{x,b}$ such that for all $n \ge n_0$:
        \begin{equation}\label{eq:cor:excursion_bound}
            \sigma_{Y, x}^n < r\br{n}.
        \end{equation}
    \end{corollary}
    \begin{proof}
        We can consider only paths for which all the excursions from $Y$ return to $Y$, as for each $n$ the measure of paths that escape $Y$ at the $n$-th excursion is $0$ by a similar calculation to that in the previous lemma, and so it is the intersection of a countable set of full measure sets.
        \par As $\sum_{n=1}^{\infty} e^{-\frac{\delta}{m} r\br{n}} < \infty$, we also have $\sum_{n=1}^{\infty} e^{-\frac{\delta}{m} r\br{n}} \theta < \infty$ and so the corollary is a direct result of Borel-Cantelli and the previous lemma.
    \end{proof}
    This gives us a bound on the excursion lengths of the random walk out of such a set $Y$. We will need to translate this to a result on the diagonal flow (Lemma \ref{lemma:r_bound}), as bounds on it can later be translated into Diophantine approximation results.
    \begin{definition}
        We shall denote by $\bar{\tau}_{Y, x}^n\br{b}$ the consecutive return time for the diagonal flow sampled only along $t_n$ (calculated in (\ref{eq:tn_formula})), and the excursion length by $\bar{\sigma}_{Y, x}^n$. Specifically:
        \begin{equation}
            \bar{\tau}_{Y, x}^n\br{b} = \inf\set{n > \bar{\tau}_{Y, x}^{n-1} \mid a_{t_n} u_{\pi\br{b}} x \in Y},
        \end{equation}
        where:
        \begin{equation}
            \bar{\tau}_{Y, x}^1 = \bar{\tau}_{Y, x} = \inf\set{n > 0 \mid a_{t_n} u_{\pi\br{b}} x \in Y}.
        \end{equation}
        And the excursion length is also similarly defined for $n \ge 1$ as: $\bar{\sigma}_{Y, x}^n = \bar{\tau}_{Y, x}^{n+1} - \bar{\tau}_{Y, x}^{n}$, and for $n = 0$, it is simply: $\bar{\sigma}_{Y, x}^0 = \bar{\tau}_{Y, x}^{1}$
    \end{definition}
    \begin{corollary}\label{cor:diagonal_result}
        Fix $x \in X$ and some monotone increasing $r: \mathbb{N} \rightarrow \mathbb{R}$ such that $\sum_{n=1}^{\infty} e^{-\frac{\delta}{m} r\br{n}} < \infty$. There exists a compact set $Y' \subseteq X$ such that for $\beta$-a.e. $b \in B$, exists some $n_1$ such that for all $n \ge n_1$:
        \begin{equation}
            \bar{\sigma}_{Y', x}^n \le r\br{n}.
        \end{equation}
    \end{corollary}
    \begin{proof}
        Take $Y$ to be as before. By Lemma \ref{lemma:bdd_dist} the diagonal flow remains bounded distance from the random walk, we shall define by some compact neighborhood of the identity $O$. Choose $Y' = \bigcup_{y \in Y} Oy$ to be a bounded compact set containing these neighborhoods for all the elements of $Y$ (which is compact). Denote by $B'$ the full measure set from Corollary \ref{cor:excursion_bound} for which bound (\ref{eq:cor:excursion_bound}) holds. Take some $b \in B'$ and the corresponding $n_0$. Denote by $n_1 = \tau_{Y, x}^{n_0}$. Take some $n \ge n_1$ and consider the $n$-th excursion out of $Y'$ (possibly of length 1 if we remain in $Y'$). As each excursion is at least of length 1, the initial time ($\bar{\tau}_{Y, x}^{n}$) is greater than $\tau_{Y, x}^{n_0}$.
        \par Notice that by the construction of $Y'$, if the random walk returns to $Y$, then the diagonal flow must return to $Y'$. Therefore, excursions out of $Y'$ of the diagonal flow are fully contained in excursions out of $Y$ of the random walk, but this is in a ``time-wise'' sense as the diagonal flow is continuous as opposed to the discrete random walk. If we denote by $t_s$ the starting time and by $t_e$ the end time of the diagonal flow excursion, then there exists an excursion of the random walk, whose index we will denote by $n'$, such that:
        \begin{equation}
            \tau_{Y, x}^{n'} \le t_s \le t_e \le \tau_{Y, x}^{n'+1}
        \end{equation}
        \par Therefore, by the above it is clear that both: $n \ge n'$ (as we cannot return to $Y'$ without returning to $Y$) and $\bar{\sigma}_{Y', x}^n \le \sigma_{Y, x}^{n'} \le r\br{n'}$. As $r$ is monotone, we have:
        \begin{equation}
            \bar{\sigma}_{Y', x}^n \le \sigma_{Y, x}^{n'} \le r\br{n'} \le r\br{n}.
        \end{equation}
    \end{proof}
}
\subsection{Minimal Vectors Bounds Along an Excursion}
{
    We shall now show how a bound on the length of the minimal non zero vector can be obtained from the bound on the excursion length. We'll denote by $\Delta$ the $\sup$ norm of the minimal vector in the lattice, set:
	\begin{equation}\label{eq:l_def}
		l\br{x} = \log \Delta^{-1}\br{x},
	\end{equation}
	and write:
    \begin{equation}
        \nu_{Y',x}^n \br{b} = \max\limits_{t_{\bar{\tau}_{Y', x}^{n}\br{b}} \le t < t_{\bar{\tau}_{Y', x}^{n+1}\br{b}}} l \br{ a_{t} u_{\pi\br{b}} x},
    \end{equation}
    which is the maximum along the $n$-th excursion of the diagonal flow, when sampling the excursions at times $t_i$, but the maximum is for all times.
    \begin{lemma}\label{lemma:growth_bound}
        Let $Y'$ be as in Corollary \ref{cor:diagonal_result}. There exists $Q > 0$ depending on $Y'$ such that we have the following bound on $\nu_{Y',x}^n \br{b}$:
        \begin{equation}
            \nu_{Y',x}^n \br{b} \le -\bar{\sigma}_{Y', x}^n\br{b}\frac{d\log \kappa}{\br{d+1}^2} + Q.
        \end{equation}
    \end{lemma}
    \begin{proof}
        First notice that $Y'$ is bounded, and therefore by Mahler's Criterion (see e.g. \cite{kleinbock_survey} - Kleinbock survey, Theorem 2.1) $\Delta$ is bounded from below on $Y'$. We'll denote by $\varepsilon = \inf_{y \in Y'} \Delta \br{y} > 0$. Notice that by the definition of $\bar{\tau}$, if we denote $n_s = \bar{\tau}_{Y', x}^{n}\br{b}$ and $n_e = \bar{\tau}_{Y', x}^{n+1}\br{b}$ then we have $l\br{a_{t_{n_s}} u_{\pi\br{b}} x} \le -\log \varepsilon$ and $l\br{a_{t_{n_e}} u_{\pi\br{b}} x} \le -\log \varepsilon$.
        \par Next, notice that $n_e - n_s = \bar{\sigma}_{Y', x}^n =: \sigma$ and that $a_{t_{n_e}} u_{\pi\br{b}} x$ differs from $a_{t_{n_s}} u_{\pi\br{b}} x$ exactly by applying the map $a_{T}$, for $T := t_{n_e} - t_{n_s}$. Recall we had the formula (\ref{eq:tn_formula}) that states: $t_n = -n\frac{d\log \kappa}{d+1}$, so:
        \begin{equation}
            T = t_{n_e} - t_{n_s} = -\br{n_e-n_s}\frac{d\log \kappa}{d+1} = -\sigma\frac{d\log \kappa}{d+1}.
        \end{equation}
        Assume that the \textbf{minimal} minimal vector along the excursion is attained at time $t' = t_{n_s}+t$, and denote it by $v$. It is the image of some vector in the lattice at time $t_{n_s}$, which is $a_{t' - t_{n_s}}^{-1} v = a_{-t} v$, and has an image in the lattice at time $t_{n_e}$,  which is $a_{t_{n_e} - t'} v = a_{T - t} v$. We have that $\norm{a_{-t} v} \ge \varepsilon$ and $\norm{a_{T - t} v} \ge \varepsilon$ as these are vectors of lattices in $Y'$, for which the minimal vector's norm is bounded below by $\varepsilon$ (and these aren't necessarily minimal).
        \par Notice that $a_{-t}$ can expand the vector by at most $e^{t/d}$, and then $a_{T - t}$ can expand by at most $e^{T-t}$. Both must expand to a vector of norm $\varepsilon$ at least, so we can therefore give the bound:
        \begin{equation}
            \norm{v} \ge \max \set {\varepsilon e^{-\br{T-t}}, \varepsilon e^{-t/d}} \ge \varepsilon e^{-\frac{T}{d+1}},
        \end{equation}
        as when $t=\frac{d}{d+1}T$ these expressions are equal, and otherwise one of them must be greater than this value. In conclusion, as $l$ is at a maximum when $\Delta$ is at a minimum:
        \begin{multline}
            \quad\quad\quad\quad\quad\quad
            \nu_{Y',x}^n \br{b} = 
            \max\limits_{
                t_{\bar{\sigma}_{Y', x}^{n}\br{b}} \le t <
                t_{\bar{\sigma}_{Y', x}^{n+1}\br{b}}
            } l \br{ a_t u_{\pi\br{b}} x} = \\ =
            l \br{a_{t'} u_{\pi\br{b}} x} =
            \log \norm{v}^{-1} \le \\ \le
            \frac{T}{d+1} - \log\varepsilon = 
            -\sigma\frac{d\log \kappa}{\br{d+1}^2} - \log\varepsilon.
            \quad\quad\quad\quad\quad\quad
        \end{multline}
        Taking $Q = -\log\varepsilon$ gives the desired result.
    \end{proof}
    \begin{corollary}
        Fix $x \in X$ and some monotone increasing $r: \mathbb{N} \rightarrow \mathbb{R}$ such that:
        \begin{equation}
            \sum_{n=1}^{\infty} \exp\br{{\frac{\delta\br{d+1}^2}{m d \log \kappa} r\br{n}}} < \infty.
        \end{equation}
        There exists a compact set $Y \subseteq X$ such that for $\beta$-a.e. $b \in B$, exists some $n_0$ such that for all $n \ge n_0$:
        \begin{equation}
            \nu_{Y, x}^n \le r\br{n}.
        \end{equation}
    \end{corollary}
    \begin{proof}
        Denote $\bar{r}\br{n} = - \frac{\br{d+1}^2}{d\log\kappa} \br{r\br{n} - Q}$. $\bar{r}$ is also monotone increasing as $\log \kappa < 0$. Notice:
        \begin{multline}
            \quad\quad\quad\quad\quad\quad
            \sum_{n=1}^{\infty} \exp\br{-\frac{\delta}{m} \bar{r}\br{n}} =
            \sum_{n=1}^{\infty} \exp\br{\frac{\delta\br{d+1}^2}{m d \log \kappa} r\br{n} - Q\frac{\delta\br{d+1}^2}{m d \log \kappa}} = \\ = 
            \exp\br{-Q\frac{\delta\br{d+1}^2}{m d \log \kappa}} \sum_{n=1}^{\infty} \exp\br{\frac{\delta\br{d+1}^2}{m d \log \kappa} r\br{n}}
            < \infty.
            \quad\quad\quad\quad\quad\quad
        \end{multline}
        We can therefore use it in Corollary \ref{cor:diagonal_result}, and take $n_0 = n_1$, $Y = Y'$. From Corollary \ref{cor:diagonal_result} we have that for all $n \ge n_0$ $\bar{\sigma}_{Y, x}^n \le \bar{r}\br{n}$. Therefore using Lemma \ref{lemma:growth_bound} we get:
        \begin{equation}
            \nu_{Y,x}^n \br{b} \le 
            -\bar{\sigma}_{Y, x}^n\br{b}\frac{d\log \kappa}{\br{d+1}^2} + Q \le 
            -\bar{r}\br{n}\frac{d\log \kappa}{\br{d+1}^2} + Q =
            r\br{n}.
        \end{equation}
    \end{proof}
    Recall from Corollary \ref{cor:m_0}: $\delta \in \br{0, -\log \br{CA_m}}$, from Definition \ref{def:ch} (3): $A_m = c\kappa^{m \alpha}$ and from Theorem \ref{thm:ch}: $\alpha = \frac{\varrho \varpi}{d+1}$ (for $0 < \varrho < 1$). Combining these, we get:
    \begin{equation}
        0 < \delta < -\log \br{CA_m} = -\log \br{Cc\kappa^{m \alpha}} = -\log \br{Cc\kappa^{m \frac{\varrho \varpi}{d+1}}} = -m \frac{\varrho \varpi}{d+1}\log \kappa - \log Cc.
    \end{equation}
    First take $\varepsilon > 0$, and set $\varepsilon' = \frac{\varepsilon}{2}$ and choose $\varrho = 1 - \varepsilon'$, and then choose:
    \begin{equation}
        \delta = -m \frac{\varrho \varpi}{d+1}\log \kappa - \log Cc - \eta = -m \frac{\br{1-\varepsilon'} \varpi}{d+1}\log \kappa - \log Cc - \eta,
    \end{equation}
    where $\eta \in \br{0,1}$ is small enough so that $\delta$ is within the required range. So we have:
    \begin{multline}
        \frac{\delta\br{d+1}^2}{m d \log \kappa} = \br{-m \frac{\br{1-\varepsilon'} \varpi}{d+1}\log \kappa - \log Cc - \eta}\frac{\br{d+1}^2}{m d \log \kappa} = \\ =
        -\frac{\varpi\br{d+1}}{d}\br{1-\varepsilon'} + \frac{1}{m}\br{\frac{\br{d+1}^2}{-d \log\kappa}\br{\log Cc + \eta}} \le \\ \le
        \frac{\varpi\br{d+1}}{d}\br{1-\varepsilon'} + \frac{1}{m}\br{\frac{\br{d+1}^2}{-d \log\kappa}\br{\log Cc + 1}}
        = -\frac{\varpi\br{d+1}}{d}\br{1-\varepsilon'} + \frac{D}{m}.
    \end{multline}
    Notice $\log \kappa < 0$, and denoting by $D := \br{\frac{\br{d+1}^2}{-d \log\kappa}\br{\log Cc + 1}}$ we see $D$ is a constant (i.e., does not depend on $m$, $\rho$, but only on the IFS). Next choose $m > m_0$ (recall $m_0$ from Corollary \ref{cor:from_lemma_63}) large enough such that $\frac{D}{m} < \varepsilon'\frac{\varpi\br{d+1}}{d}$, which is possible as $\frac{\varpi\br{d+1}}{d}$ is also a constant of the IFS. Then we get:
    \begin{equation}
        \frac{\delta\br{d+1}^2}{m d \log \kappa} \le -\frac{\varpi\br{d+1}}{d}\br{1-\varepsilon'} + \frac{D}{m} \le -\frac{\varpi\br{d+1}}{d}\br{1-2\varepsilon'} = 
        -\frac{\varpi\br{d+1}}{d}\br{1-\varepsilon}.
    \end{equation}
    This gives us the following result, which is no longer dependent on $\varrho$ or $m$:
    \begin{lemma}\label{lemma:dicrete_time_bound}
        For any $0 < \gamma < \frac{\varpi\br{d+1}}{d}$ and any $x \in X$ if $r: \mathbb{N} \rightarrow \mathbb{R}$ is monotone increasing such that:
        \begin{equation}
            \sum_{n=1}^{\infty} \exp\br{-\gamma r\br{n}} < \infty,
        \end{equation}
        then there exists a compact set $Y \subseteq X$ such that for $\beta$-a.e. $b \in B$, exists some $n_0$ such that for all $n \ge n_0$:
        \begin{equation}
            \nu_{Y, x}^n \le r\br{n}.
        \end{equation}
    \end{lemma}
    Lemma \ref{lemma:dicrete_time_bound} gives us a bound for excursions along the trajectory, sampled at discrete times. In Lemma \ref{lemma:r_bound} we give a result for continuous time, and pass to a bound on on the function $l$ (see (\ref{eq:l_def})) at every time point, rather than the maximal value along an excursion.
    \begin{lemma}\label{lemma:r_bound}
        For any $0 < \gamma < \frac{\varpi\br{d+1}}{d}$ and any $x \in X$ if $r: \mathbb{N} \rightarrow \mathbb{R}$ is monotone increasing such that:
        \begin{equation}
            \sum_{t=1}^{\infty} \exp\br{-\gamma r\br{t}} < \infty,
        \end{equation}
        then for $\beta$-a.e. $b \in B$, exists some $t_0$ such that for all $t \ge t_0$:
        \begin{equation}
            l\br{a_t u_{\pi \br{b}} x} \le r\br{t},
        \end{equation}
        where $l$ is as in $\br{\ref{eq:l_def}}$.
    \end{lemma}
    \begin{proof}
        Notice that because $r$ is monotone, by using the integral test for convergence, passing first to the convergence of the integral and then back to the convergence of a series, we have that:
        \begin{equation}
            \sum_{n=1}^{\infty} \exp\br{-\gamma r\br{-n\frac{d\log \kappa}{d+1}}} < \infty.
        \end{equation}
        We'll define $g\br{n} = r\br{-n\frac{d\log \kappa}{d+1}}$. So we can use the previous lemma on $n$, giving us a compact set $Y \subseteq X$ such that for $\beta$-a.e. $b \in B$, exists some $n_0$ such that for all $n \ge n_0$:
        \begin{equation}
            \nu_{Y, x}^n \le g\br{n}.
        \end{equation}
        Take $t_0 = t_{n_0}$ and fix some $t \ge t_0$. It belongs to some excursion, which we shall index by $n_1$. Since $\nu_{Y, x}^{n_1}$ gives a bound for all the excursion, and by monotonicity of $r$, as $t_n = -n\frac{d\log \kappa}{d+1}$:
        \begin{equation}
            l \br{a_t u_{\pi \br{b}} x} \le \nu_{Y, x}^{n_1} \le g\br{n_1} = r\br{-n_1\frac{d\log \kappa}{d+1}} = r\br{t_{n_1}} \le r \br{t}.
        \end{equation}
    \end{proof}
    This concludes the proof of a bound on the rate of escape for the diagonal flow, which we must now translate into a Diophantine approximation result in Chapter \ref{sec:main_computation}.
}
\section{Proof appendix}\label{appendix}
{
    Proof of Lemma \ref{lemma:bdd_dist}:
    \begin{proof}
        This Lemma can be deduced from the proof of Lemma 11.3 in \cite{sw}. We shall repeat the proof here, and emphasize the additional conclusions. This proof is dependent on additional notations, definitions and lemmas from \cite{sw} chapter 10.1 which I will not repeat here.
        \par Decompose $h_n = h_{b_1^n} = a_{t_n} k_n u_{\alpha_n}$ for some $t_n \in  \mathbb{R}$, $k_n \in K$ and $\alpha_n \in \mathcal{M} = \mathbb{R}^d$. Also write $\beta_n = \pi\br{T^n b}$, and let $\bar{h}_n = u_{-\beta_n} a_{t_n} k_n u_{\pi\br{b}}$. Obviously $\bar{h}_n$ and $h_n$ agree on their projections to $AK$, and on the other hand, letting them act on $\mathcal{M}$ via the isomorphism $\iota: \mathcal{M} \rightarrow P / AK$, we have (recall $\rho$ from Chapter \ref{ifs_translation}):
        \begin{multline}
            \rho\br{\bar{h}_n^{-1}}\br{\beta_n} = \rho\br{u_{\pi\br{b}}^{-1} k_n^{-1} a_{t_n}^{-1} u_{-\beta_n}^{-1}} \br{\beta_n} = \rho\br{u_{\pi\br{b}}^{-1} k_n^{-1} a_{t_n}^{-1}} \br{0} = \rho\br{u_{\pi\br{b}}^{-1}}\br{0} = \\
            = \pi\br{b} = \phi_{b^1_n} \br{\beta_n} = \rho\br{h_n^{-1}} \br{\beta_n}.
        \end{multline}
        So $\bar{h}_n = h_n$, and thus $\bar{h}_n x_0 = h_n x_0$. Since $\Phi$ is strictly contracting, the limit set $\mathcal{K}$ is compact, and so the sequence $\br{\beta_n}_{n\in \mathbb{N}}$ is bounded, and specifically one can find a common bound for all $b$. Since $K$ is also compact, this shows that the distance from $\bar{h}_n x_0$ to $a_t u_{\pi\br{b}} x_0$ is bounded by a number independent of $n$, $b$ and clearly of $x_0$ as evident from the definition (Definition \ref{def:bdd_dist}). Since the sequence $\br{a_{t_n}}_{n \in \mathbb{N}}$ has bounded gaps in $\br{a_t}_{t \ge 0}$, we have the required result.
    \end{proof}
}
\section{The Main Computation}\label{sec:main_computation}
{
    We shall now translate the bound on the rate of escape for the diagonal flow from Lemma \ref{lemma:r_bound} into a Khinchine style result, using the Dani-Correspondence \cite{km99} (chapter 8). Notice a slight difference in notation: in \cite{km99} the approximation is done for matrices, where here we handle vectors, i.e. $n=1, m=d$ in their notation. Additionally, when testing against $\psi$, \cite{km99} raise the norm by a power of $m$. We shall therefore reformulate the result in our notation:
    \begin{lemma}\label{lemma:dani}
        Fix some $x_0 > 0$. Let $\psi:\segco{x_0, \infty} \rightarrow \br{0, \infty}$ be a non-increasing continuous function. Then there exists a unique continuous function $r:\segco{t_0, \infty} \rightarrow \mathbb{R}$, where $t_0 = \frac{d}{d+1}\log x_0 - \frac{1}{d+1}\log\psi\br{x_0}$ such that
        \begin{enumerate}
            \item\label{item:dani1} the function $t \mapsto t - r\br{t}$ is strictly increasing and tends to $\infty$ as $t \rightarrow \infty$.
            \item\label{item:dani2} the function $t \mapsto \frac{t}{d} + r\br{t}$ is nondecreasing.
            \item\label{item:dani3} $\psi\br{e^{t - r\br{t}}} = e^{-\frac{t}{d} - r\br{t}}$ for all $t \ge t_0$.
        \end{enumerate}
        Conversely, given $t_0 \in \mathbb{R}$ and a continuous function $r:\segco{t_0, \infty} \rightarrow \mathbb{R}$ such that $\br{\ref{item:dani1}} \& \br{\ref{item:dani2}}$ hold, there exists a continuous non-increasing function $\psi : \segco{x_0, \infty} \rightarrow \br{0,\infty}$, with $x_0 = e^{t_0 - nr\br{t_0}}$ satisfying $\br{\ref{item:dani3}}$. Furthermore, for every non negative integer $q$,
        \begin{equation}
        	I_1 := \int_{x_0}^{\infty} \br{\log x}^q \psi\br{x} \D x < \infty \quad \Leftrightarrow \quad
        	I_2 := \int_{t_0}^{\infty} t^q e^{-\br{d+1}r\br{t}} \D t < \infty.
        \end{equation}
    \end{lemma}
    \begin{theorem}\label{theorem:dani}(Dani Correspondence)
    	Let $\psi : \mathbb{N} \rightarrow \br{0, \infty}$, and $l$ as in $\br{\ref{eq:l_def}}$. Then $\mathbf{x} \in \mathbb{R}^d$ is $\psi$-approximable iff there exist arbitrarily large positive $t$ such that
    	\begin{equation}
    		l\br{a_t u_{\mathbf{x}}} \ge r\br{t},
    	\end{equation}
    	where $r$ is as in Lemma \ref{lemma:dani}.
    \end{theorem}
    \begin{corollary}\label{cor:result_for_r}
    	Let $0 < \gamma < \frac{\varpi\br{d+1}}{d}$, $r: \mathbb{N} \rightarrow \mathbb{R}$ monotone increasing such that:
        \begin{equation}
            \sum_{t=1}^{\infty} \exp\br{-\gamma r\br{t}} < \infty
        \end{equation}
        and $\psi$ be the corresponding function from Lemma \ref{lemma:dani}. Then $\mu_\mathcal{K}$-a.e. $\mathbf{x} \in \mathcal{K}$ is not $\psi$-approximable.
    \end{corollary}
    \begin{proof}
    	By Lemma \ref{lemma:r_bound}, taking $x \in X$ to be the identity, we get that for $\beta$-a.e. $b \in B$, there exists some $t_0$ such that for all $t \ge t_0$:
        \begin{equation}
            l\br{a_t u_{\pi \br{b}}} = l\br{a_t u_{\pi \br{b}} x} \le r\br{t},
        \end{equation}
        and so we can immediately conclude that $\pi\br{b}$ is not $\psi$-approximable by Theorem \ref{theorem:dani}. As $\mu_\mathcal{K}$ is the pushforward of $\beta$ under the coding map (Definition \ref{def:bernoulli}), we get the desired result.
    \end{proof}
    We would like to translate this requirement into a requirement on $\psi$ rather than $r$. This could be done by passing to integrals (which can be done as these functions are monotonic) and changing variables.
    \begin{proof}[Proof of Theorem \ref{theorem:main}]
    	First we shall extend $\psi$ into a continuous differentiable monotone decreasing function on $\mathbb{R}$. As $\psi$ is monotonic, we have that:
    	\begin{equation}
    		\sum_{x=1}^{\infty} x^{\frac{\alpha}{d}-1}\psi^{\alpha}\br{x} \D x < \infty
    		\quad \Leftrightarrow \quad
    		\int_{x_0}^{\infty} x^{\frac{\alpha}{d}-1}\psi^{\alpha}\br{x} \D x < \infty.
    	\end{equation}
    	Let $r, x_0$ are as in Lemma \ref{lemma:dani}. We shall use the following change of variables:
    	\begin{equation}
    		x\br{t} = e^{t-r\br{t}},
    	\end{equation}
    	which also gives us by Lemma \ref{lemma:dani} item \ref{item:dani3} that:
    	\begin{equation}
    		\psi\br{x\br{t}} = e^{-\frac{t}{d}-r\br{t}}.
    	\end{equation}
    	Notice that:
    	\begin{equation}
    		\frac{\D}{\D t} x\br{t} = \br{1-r'\br{t}} e^{t-r\br{t}}.
    	\end{equation}
    	So, by changing variables, we get that:
    	\begin{equation}\label{eq:changing_x_to_t}
    		\int_{x_0}^{\infty} x^{\frac{\alpha}{d}-1}\psi^{\alpha}\br{x} \D x < \infty
    		\quad \Leftrightarrow \quad
    		\int_{t_0}^{\infty} x\br{t}^{\frac{\alpha}{d}-1}\psi^{\alpha}\br{x\br{t}} x'\br{t} \D t < \infty.
    	\end{equation}
    	We shall now simplify integral in RHS of (\ref{eq:changing_x_to_t}):
    	\begin{multline}
    		\int_{t_0}^{\infty} x\br{t}^{\frac{\alpha}{d}-1}\psi^{\alpha}\br{x\br{t}} x'\br{t} \D t = 
    		\int_{t_0}^{\infty} e^{\br{t-r\br{t}}\br{\frac{\alpha}{d}-1}} e^{\alpha\br{-\frac{t}{d}-r\br{t}}} \br{1-r'\br{t}} e^{t-r\br{t}} \D t =\\=
    		\int_{t_0}^{\infty} \br{1-r'\br{t}} \exp\br{\br{t-r\br{t}}\br{\frac{\alpha}{d}-1} + \alpha\br{-\frac{t}{d}-r\br{t}} + t-r\br{t}} \D t =\\=
    		\int_{t_0}^{\infty} \br{1-r'\br{t}} \exp\br{\frac{\alpha}{d}t - t - \frac{\alpha}{d}r\br{t} + r\br{t} - \frac{\alpha}{d}t - \alpha r\br{t} + t - r\br{t}} \D t =\\=
    		\int_{t_0}^{\infty} \br{1-r'\br{t}} \exp\br{- \frac{\alpha}{d}r\br{t} - \alpha r\br{t}} \D t = 
    		\int_{t_0}^{\infty} \br{1-r'\br{t}} \exp\br{- \frac{\alpha\br{d+1}}{d}r\br{t}} \D t =\\=
    		\int_{t_0}^{\infty} \exp\br{- \frac{\alpha\br{d+1}}{d}r\br{t}} - r'\br{t} \exp\br{- \frac{\alpha\br{d+1}}{d}r\br{t}} \D t =\\=
    		\int_{t_0}^{\infty} \exp\br{- \frac{\alpha\br{d+1}}{d}r\br{t}} \D t + \left[ \frac{d}{\alpha\br{d+1}} \exp\br{- \frac{\alpha\br{d+1}}{d}r\br{t}} \right]_{t=t_0}^{\infty}.
    	\end{multline}
    	As the second term in the last line is finite (because $r$ is monotone increasing), we have that:
    	\begin{equation}
    		\int_{t_0}^{\infty} x\br{t}^{\frac{\alpha}{d}-1}\psi^{\alpha}\br{x\br{t}} x'\br{t} \D t < \infty
    		\quad \Leftrightarrow \quad
    		\int_{t_0}^{\infty} \exp\br{- \frac{\alpha\br{d+1}}{d}r\br{t}} \D t < \infty.
    	\end{equation}
    	Notice that as $0 < \alpha < \varpi$, we get that $0 < \frac{\alpha\br{d+1}}{d} < \frac{\varpi\br{d+1}}{d}$, and so $r$ satisfies the requirement of Corollary \ref{cor:result_for_r}, which gives us the required result, i.e. that $\mu_\mathcal{K}$-a.e. $\mathbf{x} \in \mathcal{K}$ is not $\psi$-approximable.
    \end{proof}  
}
\section{Calculation of the Constants for Cantor Set Products}
{
    In this section, for our fractal $\mathcal{K}$, we shall be considering a very specific measure - the Hausdorff measure. Let $s = \operatorname{dim}_H \br{\mathcal{K}}$, the measure we will choose is $\mu_{\mathcal{K}}$, the restriction to $\mathcal{K}$ of the $s$-dimensional Hausdorff measure.
    \par In his paper \cite{khalil}, Khalil shows that for $\mathcal{K} = \mathcal{C} \times \mathcal{C}$, the constant $\alpha_1 = \frac{\log 2}{\log 3}$, which in turn gives the result that $\varpi = \frac{2\log 2}{\log 3}$ (as one can easily show $\alpha_2 = \frac{2\log 2}{\log 3}$ as well).
    We shall prove the more general result, Lemma \ref{lemma:consts}.
    \begin{notation}
        Let $d \ge 1$. Consider the IFS $\mathcal{F}$ on $\mathbb{R}^d$ given by the maps of the form:
        \begin{equation}
            h_v\br{\mathbf{x}} = \frac{\mathbf{x} + v}{3}, \quad\quad v\in E := \set{0,2}^d.
        \end{equation}
        The limit set $\mathcal{K}$ coincides with the product of cantor sets $\mathcal{C}^d$. Let $\mu$ be the measure on $\mathcal{K}$ derived from the Hausdorff measure (recall Example \ref{example:hausdorff}).
    \end{notation}
    
    First we shall prove that:
    \begin{lemma}
        For $\mathcal{K} = \mathcal{C}^d$ as before, $\alpha_l\br{\mu} \ge l \frac{\log 2}{\log 3}$.
    \end{lemma}
    From which we can immediately deduce that:
    \begin{equation}\label{eq:consts:upper_bound}
        \varpi = \min_{1 \le l \le d} \alpha_l\br{d - l + 1} \ge  \min_{1 \le l \le d} \frac{\log 2}{\log 3} l \br{d - l + 1} \ge d \frac{\log 2}{\log 3}.
    \end{equation}
    \begin{proof}
        Fix $\mathcal{L} = C^{d-l} \times \set{0}^{l}$. This is clearly an affine subspace of dimension $l$. Let $n \in \mathbb{N}$, and notice that for every $\varepsilon \in \segoc{\frac{1}{3^{n+1}}, \frac{1}{3^n}}$, $\mathcal{L}^{\br{\varepsilon}}$ the open $\varepsilon$-neighborhood of $\mathcal{L}$ contains all of the points corresponding to words with a prefix of $2^{n+1}$ of elements from $E^{d-l}\times\set{0}^{l}$, and no points outside the set of points corresponding to words with a prefix of $2^{n}$ of elements from $E^{d-l}\times\set{0}^{l}$. Therefore, $\frac{1}{2^{l\br{n+1}}} \le \mu \br{\mathcal{L}^{\br{\varepsilon}}} \le \frac{1}{2^{ln}}$, and therefore:
        \begin{equation}\label{eq:alpha:lower_bound}
            \alpha_l\br{\mu} \ge \liminf\limits_{\varepsilon \rightarrow 0} \frac{
                \log \mu \br{\mathcal{L}^{\br{\varepsilon}}}}
                {\log\br{\varepsilon}} = l \frac{\log 2}{\log 3}.
        \end{equation}
    \end{proof}
    \begin{proof}[Proof of Lemma \ref{lemma:consts}]
        To complete the proof, it suffices to show that $\alpha_1 \le \frac{\log 2}{\log 3}$, as then we have:
        \begin{equation}\label{eq:consts:lower_bound}
            \varpi = \min_{1 \le l \le d} \alpha \br{l}\br{d - l + 1} \le  \alpha_1 d \le d \frac{\log 2}{\log 3}.
        \end{equation}
        Together with (\ref{eq:consts:upper_bound}) we get $\varpi = d \frac{\log 2}{\log 3}$.
    \end{proof}
    \begin{lemma}\label{lemma:consts:upper_bound}
        For $\mathcal{K} = \mathcal{C}^d$ as before, $\alpha_1 \le \frac{\log 2}{\log 3}$.
    \end{lemma}
    Let $\mathcal{L} \in \mathcal{A}\br{d, d-1}$. To prove this, we shall use an inductive argument on $d$. We shall be interested in covering $\mathcal{L}^\varepsilon$ with cubes of the form: $h_v\br{I^d}$ for $v \in E^n$ and $I = \seg{0,1}$. This will give us an upper bound on $\mu\br{L^\varepsilon}$, as for each such cube is $\mu\br{h_v\br{I^d}} = \frac{1}{2^{dn}}$, and then we can easily bound $\mu\br{L^\varepsilon}$ with the measure of the covering.
    \begin{lemma}
        There exists a constant  $C_d$ dependent only on the dimension $d$ such that for every affine hyperplane $\mathcal{L} \subset \mathcal K$, $n \in \mathbb{N}$, $\varepsilon = \frac{1}{3^n}$, we can find a covering of $\mathcal{L}^\varepsilon$ with at most $C_d 2^{\br{d-1}n}$ cubes of the form $h_v\br{I^d}$ for $v \in E^n$.
    \end{lemma}
    \begin{proof}
        We shall begin with the base argument, for $d=1$. This case is quite simple, as $\mathcal{L}$ is a single point, which we can denote by $x$, and $\mathcal{L}^\varepsilon = \br{x-\varepsilon, x+\varepsilon}$. Clearly, such a segment, of length $2\varepsilon$ can intersect at most 3 disjoint segments of length $\varepsilon$. As all the segments (``cubes") $h_v\br{I}$ for $v \in E^n$ are of length $\frac{1}{3^n} = \varepsilon$, and are disjoint, we can choose $C_1 = 3$. One can even notice that every two segments are separated by a distance of at least $\varepsilon$, so we can event take $C_1 = 2$.
        \par Next, we shall continue with the induction step, with $d \ge 2$. We'll denote the coordinates in $\mathbb{R}^d$ by $x_i$ for $1 \le i \le d$, and $\mathcal{L}$'s equation by:
        \begin{equation}
            \mathcal{L} = \set{\mathbf{x} \mid \sum_{i=1}^d a_i x_i = d}.
        \end{equation}
        We shall observe ``slices" of cubes for a specific coordinate, where a slice is the set of cubes that have a common projection onto the given coordinate. We can assume WLOG that the coefficients of $\mathcal{L}$ are descending in size (absolute value), i.e. $\abs{a_1} \ge \abs{a_2} \ge ... \ge \abs{a_d}$. We shall look at slices with respect to coordinate $x_d$. Notice that only $2^n$ of $3^n$ of the slices have positive mass, as the projection onto $x_i$ coincides with cantor set, and so only the slices that coincide with the $n$-th step cubes of the one dimensional case have positive mass.
        \par Fix such a slice, $S = \seg{t, t+\varepsilon}$ (where $t = \frac{a}{3^n}$ for some $a \in \mathbb{N}$). We would like to count how many of the cubes are required to cover $\mathcal{L}^{\varepsilon}\cap\mathbb{R}^{d-1}\times S$. Observe $\mathcal{L}' = \set{\mathbf{x} \mid \sum_{i=1}^{d-1} a_i x_i = d- a_d t} \subset \mathbb{R}^{d-1}$. In this lower dimensional setting, we can use the previous step induction, to find a covering $A''$ of $\mathcal{L}'^{\br{3\varepsilon}}$ with at most $C_{d-1}2^{\br{d-2}\br{n-1}}$ dimensional cubes of dimension $d-1$ and size $3\varepsilon$. Notice each such cube can be broken up into $3^{d-1}$ cubes of size $\varepsilon$ out of which only $2^{d-1}$ have positive mass, so we can consider this to be a covering $A'$ of $C_{d-1}2^{d-1}2^{\br{d-2}\br{n-1}} = C_{d-1}2^{\br{d-2}n+1}$ cubes of size $\varepsilon$.
        \par Take $A = \set{B \times S \mid S \in A'}$. Notice that all of $A$'s elements are cubes of size $\varepsilon$ in our slice: they clearly project to $S$ on $x_d$, and as they where cubes of size $\varepsilon$ in $\mathbb{R}^{d-1}$ and $S$ is a segment of length $\varepsilon$ in $\mathcal{C}$, then they are cubes in the product space. We shall prove that $A$ covers $\mathcal{L}^{\varepsilon}\cap\mathbb{R}^{d-1}\times S$. Let $\mathbf{x} \in \mathcal{L}^{\varepsilon}\cap\mathbb{R}^{d-1}\times S$ then there exists some $\mathbf{x}' \in \mathcal{L}$ such that $d\br{\mathbf{x},\mathbf{x}'} < \varepsilon$. Clearly, as the distance is bounded by $\varepsilon$, so must be the difference in the $x_d$ coordinate, so we can deduce that $\mathbf{x}' \in \mathcal{L}\cap\mathbb{R}^{d-1}\times \br{t-\varepsilon, t+2\varepsilon}$. We would like to show that there exists some $B \in A$ such that $\mathbf{x} \in B$. Project onto $\mathbb{R}^{d-1}$. We denote the projection by $\pi$. Clearly the $d\br{\pi\br{\mathbf{x}}, \pi\br{\mathbf{x}'}} < \varepsilon$ as well. As $\mathbf{x}' \in \mathcal{L}$, we have $\sum_{i=1}^d a_i x'_i = d$, so $\sum_{i=1}^{d-1} a_i x'_i = d - a_d x'_d$. Therefore, we have:
        \begin{equation}
            d\br{\mathcal{L}', \pi\br{\mathbf{x}'}} = \frac{\abs{a_d t - a_d x'_d}}{\sqrt{\sum_{i=1}^{d-1} a_i^2}} = \frac{\abs{a_d}\abs{t - a_d x'_d}}{\sqrt{\sum_{i=1}^{d-1} a_i^2}} < \frac{\abs{a_d}2\varepsilon}{\sqrt{\sum_{i=1}^{d-1} a_i^2}} \le 2 \varepsilon,
        \end{equation}
        where the last inequality is because $a_d$ is the smallest coefficient (in absolute value, by our assumption earlier), and the equation before last is by $\mathbf{x}' \in \mathcal{L}\cap\mathbb{R}^{d-1}\times \br{t-\varepsilon, t+2\varepsilon}$.
        \par Therefore, we have that $d\br{\mathcal{L}', \pi\br{\mathbf{x}'}} < 2\varepsilon$, and $d\br{\pi\br{\mathbf{x}}, \pi\br{\mathbf{x}'}} < \varepsilon$ so $d\br{\mathcal{L}', \pi\br{\mathbf{x}}} < 3\varepsilon$, i.e. $\pi\br{\mathbf{x}} \in \mathcal{L}'^{\br{\varepsilon}}$, and therefore by our construction, there exists some $B' \in A'$ such that $\pi\br{\mathbf{x}} \in B'$. Finally, notice that $\mathbf{x} \in B'\times S \in A$, concluding the proof that $A$ covers the slice $\mathcal{L}^{\varepsilon}\cap\mathbb{R}^{d-1}\times S$.
        \par Recall that there are at most $2^n$ relevant slices with positive measure, and as within each slice, $\mathcal{L}^{\br{\varepsilon}}$ can be covered with at most $C_{d-1}2^{\br{d-2}n+1}$ cubes of size $\varepsilon$ as we just saw. Therefore we can cover $\mathcal{L}^{\br{\varepsilon}}$ entirely with $2^n C_{d-1}2^{\br{d-2}n+1} = 2 C_{d-1}2^{\br{d-1}n}$ cubes. Taking $C_d := 2C_{d-1}$ concludes our proof.
    \end{proof}
    We can now prove Lemma \ref{lemma:consts:upper_bound}.
    \begin{proof}
        Let $\mathcal{L} \subset \mathcal{K}$ be an affine hyperplane. By the previous lemma, for $\varepsilon = \frac{1}{3^n}$, $\mathcal{L}^{\br{\varepsilon}}$ can be covered by at most $C_d 2^{\br{d-1}n}$ cubes of size $\varepsilon$, each of measure $\frac{1}{2^{dn}}$, and so: $\mu\br{\mathcal{L}^{\br{\varepsilon}}} \le \frac{C_d 2^{\br{d-1}n}}{2^{dn}} = \frac{C_d}{2^n}$. Therefore:
        \begin{multline}
            \alpha_1\br{\mu} := \liminf\limits_{\varepsilon \rightarrow 0} \frac{\log \sup_{\mathcal{L} \in \mathcal{A}\br{d, d-1}} \mu \br{\mathcal{L}^{\br{\varepsilon}}}}{\log\br{\varepsilon}} \le \liminf\limits_{n \rightarrow \infty} \frac{\log \sup_{\mathcal{L} \in \mathcal{A}\br{d, d-1}} \mu \br{\mathcal{L}^{\br{\frac{1}{3^n}}}}}{\log\br{\frac{1}{3^n}}} \le \\ \le \liminf\limits_{n \rightarrow \infty} \frac{\log\br{\frac{C_d}{2^n}}}{\log\br{\frac{1}{3^n}}} = \liminf\limits_{n \rightarrow \infty} \frac{n \log 2 - \log C_d}{n \log 3} = \frac{\log 2}{\log 3}.
        \end{multline}
    \end{proof}
    \begin{corollary}\label{cor:alpha_1}
        Let $d \ge 1$, $\mathcal{K} = \mathcal{C}^d$ and $\mu$ be the measure on $\mathcal{K}$ derived from the Haar measure. Then $\alpha_1\br{\mu} = \frac{\log 2}{\log 3}$
    \end{corollary}
    \begin{proof}
        Immediate from Lemma \ref{lemma:consts:upper_bound} and (\ref{eq:alpha:lower_bound}).
    \end{proof}
}

\bibliographystyle{alpha}
\bibliography{ref}
\end{document}